\documentclass[12pt,english,refpage, intoc, refeq]{article}
\usepackage[T1]{fontenc}
\usepackage[latin9]{inputenc}
\usepackage{xcolor}
\usepackage{pdfcolmk}
\usepackage{bm}
\usepackage{amsthm}
\usepackage{amsmath}
\usepackage{amssymb}
\usepackage{esint}
\PassOptionsToPackage{normalem}{ulem}
\usepackage{ulem}

\makeatletter

\numberwithin{equation}{section}
\numberwithin{figure}{section}
\theoremstyle{plain}
\newtheorem{thm}{\protect\theoremname}
  \theoremstyle{plain}
  \newtheorem{lem}[thm]{\protect\lemmaname}
  \theoremstyle{remark}
  \newtheorem{rem}[thm]{\protect\remarkname}
  \theoremstyle{plain}
  \newtheorem{prop}[thm]{\protect\propositionname}

\makeatother

\usepackage{babel}
  \providecommand{\lemmaname}{Lemma}
  \providecommand{\propositionname}{Proposition}
  \providecommand{\remarkname}{Remark}
\providecommand{\theoremname}{Theorem}

\begin{document}

\title{Stochastic differential equations driven by generalized grey noise}

\author{\textbf{José Luís da Silva},\\
CCM, University of Madeira, Campus da Penteada,\\
9020-105 Funchal, Portugal.\\
Email: luis@uma.pt\and \textbf{Mohamed Erraoui}\\
Université Cadi Ayyad, Faculté des Sciences Semlalia,\\
 Département de Mathématiques, BP 2390, Marrakech, Maroc\\
Email: erraoui@uca.ma}
\maketitle
\begin{abstract}
In this paper we establish a substitution formula for stochastic differential
equation driven by generalized grey noise. We then apply this formula
to investigate the absolute continuity of the solution with respect
to the Lebesgue measure and the positivity of the density. Finally,
we derive an upper bound and show the smoothness of the density.\medskip{}

\end{abstract}
\tableofcontents{}

\section{Introduction}

A number of stochastic models for explaining anomalous diffusion have
been introduced in the literature, among them we would like to quote
the fractional Brownian motion (fBm), see e.g.~\cite{MandelbrotNess1968},
\cite{Taqqu03}, the Lévy flights \cite{DSU08}, the grey Brownian
motion (gBm) \cite{Schneider90}, \cite{MR1190506}, the generalized
grey Brownian motion (ggBm) denoted by $B_{\alpha,\beta}$ \cite{Mura2008},
\cite{Mura_mainardi_09}, \cite{Mura_Pagnini_08} and references therein.
The latter is a family of self-similar with stationary increments
processes ($\frac{\alpha}{2}$-sssi) where the two real parameters
$\alpha\in\left(0,2\right)$ and $\beta\in\left(0,1\right]$. It includes
fBm when $\alpha\in\left(0,2\right)$ and $\beta=1$, and time-fractional
diffusion stochastic processes when $\alpha=\beta\in\left(0,1\right)$.
The gBm corresponds to the choice $\alpha=\beta$, with $0<\beta<1$.
Finally, the standard Brownian motion (Bm) is recovered by setting
$\alpha=\beta=1$. We observe that only in the particular case of
Bm the corresponding process is Markovian. Moreover the process $B_{\alpha,\beta}$
has $(\frac{\alpha}{2}-\varepsilon)$-Hölder continuous trajectories
for all $\varepsilon>0$ and it can be represented (in law) as a scale
mixture $(\sqrt{Y_{\beta}}B_{H})$ where $B_{H}$ is a standard fBm
with Hurst parameter $H=\alpha/2$ and $Y_{\beta}$ is an independent
non-negative random variable, for the details see Section \ref{sec:prelim}.

We will consider the following stochastic differential equation (SDE)
on $\mathbb{R}^{n}$ 
\begin{equation}
X_{t}=x_{0}+\sum_{j=1}^{d}\int_{0}^{t}V_{j}(X_{s})dB_{\alpha,\beta}^{j}(s)+\int_{0}^{t}V_{0}(X_{s})\, ds,\quad t\in\left[0,T\right],\label{eq:1}
\end{equation}
where $x_{0}\in\mathbb{R}^{n},$ $T>0$ is a fixed time, $\bm{B}_{\alpha,\beta}=(B_{\alpha,\beta}^{1},\ldots,B_{\alpha,\beta}^{d})$
is a $d$-dimensional ggBm, $\alpha\in\left(1,2\right)$, $\beta\in(0,1]$
and $\left\{ V_{j};0\leq j\leq d\right\} $ is a collection of vector
fields of $\mathbb{R}^{n}$.

The stochastic integral appearing in (\ref{eq:1}) is a pathwise Riemann-Stieltjes
integral, see \cite{Young1936}. It is well known that, under suitable
assumptions on $\bm{V}=(V_{1},\ldots,V_{d}),$ the equation (\ref{eq:1})
has a unique solution which is $(\frac{\alpha}{2}-\varepsilon)$-Hölder
continuous for all $\varepsilon>0$. This result was obtained in \cite{Lyons1994}
using the notion of $p$-variation. The theory of rough paths, introduced
by Lyons in \cite{Lyons1994}, was used by Coutin and Qian in order
to prove an existence and uniqueness result for the equation (\ref{eq:1})
driven by fBm, see \cite{Coutin2002}. Nualart and R\u{a}\c{s}canu
\cite{Nualart2002} have established the existence of a unique solution
for a class of general differential equations that includes (\ref{eq:1})
using the fractional integration by parts formula obtained by Zähle
for Young integral, see \cite{Zaehle1998}. 

The representation in law of $\bm{B}_{\alpha,\beta}$, see (\ref{eq:gBm_rep_ndim})
below, allows us to consider, instead of the equation (\ref{eq:1}),
the following equation 
\begin{equation}
X_{t}^{H}=x_{0}+\sum_{j=1}^{d}\int_{0}^{t}V_{j}(X_{s}^{H})d\big(\sqrt{Y_{\beta}}B_{H}^{j}\big)(s)+\int_{0}^{t}V_{0}(X_{s}^{H})\, ds,\quad t\in\left[0,T\right].\label{eq:wrep}
\end{equation}
This is due to the fact that the solutions of the SDEs (\ref{eq:1})
and (\ref{eq:wrep}) induces the same distribution on the space of
continuous functions $C\left(\left[0,T\right];\mathbb{R}^{n}\right)$.
Furthermore, since the stochastic integral in (\ref{eq:wrep}) is
a pathwise Riemann-Stieltjes integral, then the SDE (\ref{eq:wrep})
can be written as
\begin{equation}
X_{t}^{H}=x_{0}+\sqrt{Y_{\beta}}\sum_{j=1}^{d}\int_{0}^{t}V_{j}(X_{s}^{H})dB_{H}^{j}(s)+\int_{0}^{t}V_{0}(X_{s}^{H})\, ds,\quad t\in\left[0,T\right].\label{eq:eq multip}
\end{equation}

The main purpose of this paper is to establish a substitution formula
(SF) for equation (\ref{eq:eq multip}). Let us now describe our approach.
For each $y>0$, we consider the following equation 
\begin{equation}
X_{t}^{H}(y)=x_{0}+\sqrt{y}\sum_{j=1}^{d}\int_{0}^{t}V_{j}(X_{s}^{H}(y))dB_{H}^{j}(s)+\int_{0}^{t}V_{0}(X_{s}^{H}(y))\, ds.\label{eq:2}
\end{equation}
It is well known that, under suitable assumptions, see e.g.~Nualart
and R\u{a}\c{s}canu \cite{Hu2007}, that if $1-H<\lambda<\frac{1}{2}$
the SDE (\ref{eq:2}) has a strong $\left(1-\lambda\right)$-Hölder
continuous solution $X_{\cdotp}^{H}(y)$. To establish a SF, the natural
idea is to replace $y$ in (\ref{eq:2}) by the random variable $Y_{\beta}$
and prove that $X_{\cdot}^{H}(Y_{\beta})$ satisfies the SDE (\ref{eq:eq multip}).
For more details on the SF we refer to \cite{Nualart2006a}. To handle
this problem, the key is to prove, for each $t\in\left[0,T\right]$,
the following equalities 
\begin{equation}
\left.\int_{0}^{t}V_{j}(X_{s}^{H}(y))dB_{H}^{j}(s)\right|_{y=Y_{\beta}}=\int_{0}^{t}V_{j}(X_{s}^{H}(Y_{\beta}))dB_{H}^{j}(s),\quad j=1,\ldots,d,\label{eq:subs1}
\end{equation}
and 
\begin{equation}
\int_{0}^{t}V_{0}(X_{s}^{H}(y))\, ds\bigg|_{y=Y_{\beta}}=\int_{0}^{t}V_{0}(X_{s}^{H}(Y_{\beta}))\, ds.\label{eq:subs 2}
\end{equation}
 To this end we need to study the regularity of the solution $X_{t}^{H}(y)$
of the SDE (\ref{eq:2}) with respect to $y$. Once this is accomplished,
we use the SF to show the absolute continuity of the law of the solution
$X^{H}(Y_{\beta})$ and the positivity of its density $p_{X_{t}^{H}(Y_{\beta})}$.
Subsequently to give a Gaussian mixture type upper bound and to study
the smoothness of $p_{X_{t}^{H}(Y_{\beta})}$. We emphasize the fact
that these results are essentially due to those established for the
density $p_{X_{t}^{H}(y)}$ of the law of $X_{t}^{H}(y)$, see \cite{BOT14,Baudoin:2014tw,Nualart2009},
and the dependence with respect to $y$ of $p_{X_{t}^{H}(y)}$. Indeed,
using the SF and the independence of $\{X_{t}^{H}(y),0\leq t\leq T,y>0\}$
and $Y_{\beta}$, the density $p_{X^{H}(Y_{\beta})}$ is given by
\[
p_{X_{t}^{H}(Y_{\beta})}(z)=\int_{0}^{+\infty}p_{X_{t}^{H}(y)}(z)p_{Y_{\beta}}(y)\, dy,\quad z\in\mathbb{R}^{n},
\]
where $p_{Y_{\beta}}$ is the density of the law of $Y_{\beta}$.
Hence, the density $p_{X_{t}^{H}(Y_{\beta})}$ is given in terms of
a parameter dependent integral, implying that all the properties of
$p_{X_{t}^{H}(Y_{\beta})}$ will be deducted from those of $p_{X_{t}^{H}(y)}$.
This persuade us to borrow the hypotheses of the cited works to realize
the above results.

\section{Preliminaries}

\label{sec:prelim}

According to Mura and Pagnini \cite{Mura_Pagnini_08}, the ggBm $B_{\alpha,\beta}$
is a stochastic process defined on a probability space $\left(\Omega,\mathcal{F},\mathbb{P}\right)$
such that for any collection $0\leq t_{1}<t_{2}<\ldots<t_{n}<\infty$
the joint probability density function of $\left(B_{\alpha,\beta}(t_{1}),\ldots,B_{\alpha,\beta}(t_{n})\right)$
is given by 

\begin{equation}
f_{\alpha,\beta}(x,t_{1},\ldots,t_{n})=\dfrac{\left(2\pi\right)^{-\frac{n}{2}}}{\sqrt{\det(\Sigma_{\alpha})}}\int_{0}^{\infty}\dfrac{1}{\tau^{n/2}}\exp\left(-\dfrac{x^{\top}\Sigma_{\alpha}^{-1}x}{2\tau}\right)M_{\beta}(\tau)d\tau,\label{eq:density}
\end{equation}
where $n\in\mathbb{N}$, $x\in\mathbb{R}^{n}$, $\Sigma_{\alpha}=(a_{i,j})_{i,j=1}^{n}$
is the matrix given by 
\[
a_{i,j}=t_{i}^{\alpha}+t_{j}^{\alpha}-|t_{i}-t_{j}|^{\alpha},
\]
and $M_{\beta}$ is the so-called $M$-Wright probability density
function (a natural generalization of the Gaussian density) which
is related to the Mittag-Leffler function through the following Laplace
transform 
\begin{equation}
\int_{0}^{\infty}e^{-s\tau}M_{\beta}(\tau)\, d\tau=E_{\beta}(-s).\label{eq:M_wright}
\end{equation}
Here $E_{\beta}$ is the Mittag-Leffler function of order $\beta$,
defined by
\[
E_{\beta}(x)=\sum_{n=0}^{\infty}\frac{x^{n}}{\Gamma(\beta n+1)},\quad x\in\mathbb{R}.
\]
It follows from (\ref{eq:density}) that for a given $u=(u_{1},\ldots,u_{n})\in\mathbb{R}^{n}$,
$n\in\mathbb{N}$ and any collection $\{B_{\alpha,\beta,}(t_{1}),\ldots,B_{\alpha,\beta,}(t_{n})\}$
with $0\leq t_{1}<t_{2}<\ldots<t_{n}<\infty$ we have
\begin{equation}
\mathbb{E}\left(\exp\left(i\sum_{k=1}^{n}u_{k}B_{\alpha,\beta}(t_{k})\right)\right)=E_{\beta}\left(-\frac{1}{2}u^{\top}\Sigma_{\alpha}u\right).\label{eq:gBm_nG}
\end{equation}
Equation~(\ref{eq:gBm_nG}) shows that ggBm, which is not Gaussian
in general, is a stochastic process defined only through its first
and second moments which is a property of Gaussian processes.

The following properties can be easily derived from~(\ref{eq:gBm_nG}).
\begin{enumerate}
\item $B_{\alpha,\beta}(0)=0$ almost surely. In addition, for each $t\geq0$,
the moments of any order are given by
\[
\begin{cases}
\mathbb{E}(B_{\alpha,\beta}^{2n+1}(t)) & =0,\\
\mathbb{E}(B_{\alpha,\beta}^{2n}(t)) & =\frac{(2n)!}{2^{n}\Gamma(\beta n+1)}t^{n\alpha}.
\end{cases}
\]

\item For each $t,s\geq0$, the characteristic function of the increments
is 
\begin{equation}
\mathbb{E}\big(e^{iu(B_{\alpha,\beta}(t)-B_{\alpha,\beta}(s))}\big)=E_{\beta}\left(-\frac{u^{2}}{2}|t-s|^{\alpha}\right),\quad u\in\mathbb{R}.\label{eq:cf_gBm_increments}
\end{equation}

\item The covariance function has the form
\begin{equation}
\mathbb{E}(B_{\alpha,\beta}(t)B_{\alpha,\beta}(s))=\frac{1}{2\Gamma(\beta+1)}(t^{\alpha}+s^{\alpha}-|t-s|^{\alpha}),\quad t,s\geq0.\label{eq:auto-cv-gBm}
\end{equation}

\end{enumerate}
It was shown in \cite{Mura_Pagnini_08} that the ggBm $B_{\alpha,\beta}$
admits the following representation
\begin{equation}
\big\{ B_{\alpha,\beta}(t),\; t\geq0\big\}\overset{d}{=}\big\{\sqrt{Y_{\beta}}B_{H}(t),\; t\geq0\big\},\label{gbm-rep}
\end{equation}
where $\overset{d}{=}$ denotes the equality of the finite dimensional
distribution and $B_{H}$ is a standard fBm with Hurst parameter $H=\alpha/2$.
$Y_{\beta}$ is an independent non-negative random variable with probability
density function $M_{\beta}$. A process with the representation given
as in (\ref{gbm-rep}) is known to be variance mixture of normal distributions.
A consequence of the representation (\ref{gbm-rep}) is the Hölder
continuity of the trajectories of ggBm which reduces to the Hölder
continuity of the fBm. Thus we have
\begin{equation}
\mathbb{E}(|B_{\alpha,\beta}(t)-B_{\alpha,\beta}(s)|^{p})=c_{p}|t-s|^{p\alpha/2}.\label{eq:cont_gBm}
\end{equation}
We conclude that the process $B_{\alpha,\beta}$ has $(\frac{\alpha}{2}-\varepsilon)$-Hölder
continuous trajectories for all $\varepsilon>0$. So, we can use the
integral introduced by Young \cite{Young1936} with respect to $B_{\alpha,\beta}$.
That is, for any Hölder continuous function $f$ of order $\gamma$
such that $\gamma+\left(\alpha/2\right)>1$ and every subdivision
$\left(t_{i}^{n}\right)_{i=0,\ldots,T}$ of $[0,T]$, whose mesh tends
to $0$, as $n$ goes to $\infty$, the Riemann sums 
\[
\sum_{i=0}^{n-1}f(t_{i}^{n})\left(B_{\alpha,\beta}(t_{i+1}^{n})-B_{\alpha,\beta}(t_{i}^{n})\right)
\]
converge to a limit which is independent of the subdivision $\left(t_{i}^{n}\right)_{i=0,\ldots,T}$.
We denote this limit by 
\[
\int_{0}^{T}f(t)\, dB_{\alpha,\beta}(t).
\]
Till now we have recalled the ggBm in $1$-dimension, but from now
on we use a $d$-dimensional ggBm $\bm{B}_{\beta,\alpha}=(B_{\alpha,\beta}^{1},\ldots,B_{\alpha,\beta}^{d})$
($0<\beta\le1$, $1<\alpha\le2$) with characteristic function 
\[
\mathbb{E}\big(e^{i(\bm{x},\bm{B}_{\alpha,\beta}(t))_{\mathbb{R}^{d}}}\big)=E_{\beta}\left(-\frac{1}{2}(\bm{x},\bm{x})_{\mathbb{R}^{d}}t^{\alpha}\right)
\]
and the representation in law 
\begin{equation}
\bm{B}_{\alpha,\beta}(t)=\sqrt{Y_{\beta}}\bm{B}_{H}(t),\quad t\geq0,\label{eq:gBm_rep_ndim}
\end{equation}
where $Y_{\beta}$ is independent of $\bm{B}_{H}(t)$, $\bm{B}_{H}$
is a $d$-dimensional fBm with Hurst parameter $H=\alpha/2$. 

\noindent \textbf{Notations}: Throughout this paper, unless otherwise
specified we will make use of the following notations:

For $0<\lambda<1$ we denote by $C^{\lambda}\left(0,T,\mathbb{R}^{d}\right)$
the space of all $\lambda$-Hölder continuous functions $f:[0,T]\longrightarrow\mathbb{R}^{d}$,
equipped with the norm 
\[
\|f\|_{\lambda}=\|f\|_{0,T\infty}+\|f\|_{0,T,\lambda}
\]
where
\[
\|f\|_{0,T,\infty}=\underset{0\leq t\leq T}{\sup}\,\left|f(t)\right|,\qquad\qquad\|f\|_{0,T,\lambda}=\underset{0\leq s<t\leq T}{\sup}\,\dfrac{\left|f(t)-f(s)\right|}{\left|t-s\right|^{\lambda}}.
\]
For $k,n,m\in\mathbb{N}$ we denote by $C_{b}^{k}:=C_{b}^{k}(\mathbb{R}^{n},\mathbb{R}^{m})$
the space of all bounded functions on $\mathbb{R}^{n}$ which are
$k$ times continuously differentiable in Fréchet sense with bounded
derivative up to the $k$th order, equipped with the norm 
\[
\|f\|_{C_{b}^{k}}=\|f\|_{\infty}+\|Df\|_{\infty}+\cdots+\|D^{k}f\|_{\infty}<\infty.
\]
We also denote by $C_{b}^{\infty}:=C_{b}^{\infty}(\mathbb{R}^{n},\mathbb{R}^{m})$
the class of all infinitely differentiable (in Fréchet sense) bounded
functions on $\mathbb{R}^{n}$ with bounded derivatives of all orders.

\section{Substitution theorem}

\label{sec:subs}

Throughout this paper we assume that the coefficients $V_{0}$ and
$\bm{V}$ satisfy the following hypothesis
\begin{description}
\item [{(H.1)}] 
\[
V_{0}\in C_{b}^{1},\;\bm{V}\in C_{b}^{2}.
\]

\end{description}
First we give the regularity of the solution $X_{t}^{H}(y)$ of the
SDE (\ref{eq:2}) with respect to $y$. This result will be proved
using the following Fernique-type lemma due to Saussereau \cite{Saussereau2012}.
\begin{lem}
[cf.~Lemma~2.2 in \cite{Saussereau2012}]\label{lem:fernique}

(i). Let $T>0$ and $1/2<\delta<H<1$ be given. Then, for any $\tau<1/(128\left(2T\right)^{2\left(H-\delta\right)})$,
we have 
\[
\mathbb{E}\left(\exp\left(\tau\|\bm{B}_{H}\|_{0,T,\delta}^{2}\right)\right)\leq\left(1-128\tau\left(2T\right)^{2\left(H-\delta\right)}\right)^{-1/2}.
\]

(ii). For any integer $k\geq1$ we have 
\[
\mathbb{E}\left(\|\bm{B}_{H}\|_{0,T,\delta}^{2k}\right)\leq32^{k}(2T)^{2k(H-\delta)}(2k)!.
\]
\end{lem}
\begin{rem}
For any $\tau<1/(128\left(2T\right)^{2\left(H-\delta\right)})$ we
have the following tail norm estimate for $\bm{B}_{H}$:
\begin{equation}
\mathbb{P}\big[\|\bm{B}_{H}\|_{0,T,\delta}>r\big]\leq M\exp\left(-\tau r^{2}\right),\label{eq:tail-estimate_bh}
\end{equation}
where $M=\left(1-128\tau\left(2T\right)^{2\left(H-\delta\right)}\right)^{-1/2}$.
\end{rem}
The following estimate is crucial in the proof of our main result,
cf.~Theorem~\ref{thm:main} below. It is worth to notice that such
estimate was obtained and improved by Hu and Nualart \cite{Hu2007}
and, afterward, refined in Proposition~2.3 in \cite{Saussereau2012}.
\begin{prop}
\label{prop:estim sol}Let $T>0$ and $1/2<\delta<H<1$ be given.
Under Hypothesis (\textbf{H.1}) there exist a positive constant $C_{n}$
depending on $T,\delta,H,\left\Vert V_{0}\right\Vert _{C_{b}^{1}}$
and $\|\bm{V}\|_{C_{b}^{2}}$ such that
\begin{eqnarray*}
\|X^{H}(y)-X^{H}(\widetilde{y})\|_{\delta} & \leq & C_{n}\left|\sqrt{y}-\sqrt{\widetilde{y}}\right|\|\bm{V}\|_{C_{b}^{1}}\|\bm{B}_{H}\|_{0,T,\delta}\\
 &  & \times\left(1+\|\bm{B}_{H}\|_{0,T,\delta}\right)^{2/\delta}\exp\left(C_{n}\|\bm{B}_{H}\|_{0,T,\delta}^{1/\delta}\right)
\end{eqnarray*}
for all $\left|y\right|,\left|\widetilde{y}\right|\leq n$.
\end{prop}
Now we are ready to state the regularity of the solution $X_{t}^{H}(y)$
of the SDE (\ref{eq:2}) with respect to $y$. 
\begin{prop}
\label{prop:regular}Let $T>0$ and $1/2<\delta<H<1$ be given. Under
Hypothesis (\textbf{H.1}) there exist a positive $\widetilde{C}_{n}>0$
depending on $T$, $\delta$, $H$, $\|V_{0}\|_{C_{b}^{1}}$ and $\|\bm{V}\|_{C_{b}^{2}}$
such that 
\[
\mathbb{E}\left(\underset{s\leq t}{\sup}\left|X_{s}^{H}(y)-X_{s}^{H}(\widetilde{y})\right|^{4}\right)\leq\widetilde{C}_{n}\left|y-\widetilde{y}\right|^{2},\quad t\in[0,T]
\]
for all $\left|y\right|,\left|\widetilde{y}\right|\leq n$.\end{prop}
\begin{proof}
Let $t\in[0,T]$ and $\left|y\right|,\left|\widetilde{y}\right|\leq n$
be fixed. Using the estimate in Proposition \ref{prop:estim sol}
we obtain
\begin{eqnarray*}
 &  & \mathbb{E}\left(\underset{s\leq t}{\sup}\left|X_{t}^{H}(y)-X_{t}^{H}(\widetilde{y})\right|^{4}\right)\\
 & \leq & C_{n}^{4}\left|\sqrt{y}-\sqrt{\widetilde{y}}\right|^{4}\|\bm{V}\|_{C_{b}^{1}}^{4}\\
 &  & \times\mathbb{E}\left(\|\bm{B}_{H}\|_{0,T,\delta}^{4}\left(1+\left\Vert \bm{B}_{H}\right\Vert _{0,T,\delta}\right)^{8/\beta}\exp\left(4C_{n}\|\bm{B}_{H}\|_{0,T,\delta}^{1/\delta}\right)\right).
\end{eqnarray*}
It follows from the assertions (i) and (ii) of Lemma \ref{lem:fernique}
and the following Young inequality
\[
4C_{n}\|\bm{B}_{H}\|_{0,T,\delta}^{1/\delta}\leq\dfrac{2\delta-1}{2\delta}\left(\dfrac{4C_{n}}{\varepsilon}\right)^{2\delta/(2\delta-1)}+\varepsilon^{2\delta}\|\bm{B}_{H}\|_{0,T,\delta}^{2}
\]
 that, for small enough $\varepsilon$, there exist a constant $\widetilde{C}_{n}>0$
depending on $T,\delta,H,\left\Vert V_{0}\right\Vert _{C_{b}^{1}}$
and $\left\Vert \bm{V}\right\Vert _{C_{b}^{2}}$ such that 
\[
\mathbb{E}\left(\underset{s\leq t}{\sup}\left|X_{s}^{H}(y)-X_{s}^{H}(\widetilde{y})\right|^{4}\right)\leq\widetilde{C}_{n}\left|y-\widetilde{y}\right|^{2}.
\]

\end{proof}
The following proposition provides the substitution formulas (\ref{eq:subs1})
and (\ref{eq:subs 2}).
\begin{prop}
Under Hypothesis (\textbf{H.1}) the equalities (\ref{eq:subs1}) and
(\ref{eq:subs 2}) are satisfied.\end{prop}
\begin{proof}
First let's recall that, for $j=1,\ldots,d$ and any $y>0$, the Young
integrals 
\begin{equation}
\int_{0}^{T}V_{j}(X_{s}^{H}(y))\, dB_{H}^{j}(s)\label{eq:Young1}
\end{equation}
and 
\begin{equation}
\int_{0}^{T}V_{j}(X_{s}^{H}(Y_{\beta}))\, dB_{H}^{j}(s)\label{eq:Young2}
\end{equation}
 exist. Indeed, if $1-H<\lambda<\frac{1}{2}$, then for each $y>0$,
the SDE (\ref{eq:2}) has a strong $\left(1-\lambda\right)$-Hölder
continuous solution $X_{\cdotp}^{H}(y)$. Therefore, the process $X_{.}^{H}(Y_{\beta})$
has $\left(1-\lambda\right)$-Hölder continuous paths. Then the existence
of the preceding integrals follows from the Lipschitz condition of
$V_{j}$ and the Hölder continuity of the paths of $B_{H}^{j}$. As
a consequence, for any subdivision $\left(t_{k}^{n}\right)_{k=0,\ldots,n-1}$
of $[0,T]$, whose mesh tends to $0$ as $n$ goes to $\infty$, and
each $y\geq0$, the Riemann sums
\[
S_{n}^{j}\left(y\right)=\sum_{k=0}^{n-1}V_{j}(X_{t_{k}^{n}}^{H}(y))\left(B_{H}^{j}(t_{k+1}^{n})-B_{H}^{j}(t_{k}^{n})\right)
\]
and 
\[
R_{n}^{j}=\sum_{k=0}^{n-1}V_{j}(X_{t_{k}^{n}}^{H}(Y_{\beta}))\left(B_{H}^{j}(t_{k+1}^{n})-B_{H}^{j}(t_{k}^{n})\right)
\]
converge to (\ref{eq:Young1}) and (\ref{eq:Young2}), respectively.
Now to prove (\ref{eq:subs1}), it suffices to show that $S_{n}^{j}(Y_{\beta})=R_{n}^{j}$,
converge, as $n$ goes to $\infty$, to 
\[
\int_{0}^{T}V_{j}(X_{s}^{H}(y))dB_{H}^{j}(s)\bigg|_{y=Y_{\beta}}.
\]
Taking into account that the fBm with Hurst parameter $H$ has locally
bounded $p$-variation for $p>1/H$ and the regularity of the solution
$X_{t}^{H}(y)$ with respect to $y$, cf.~Proposition \ref{prop:regular},
then the above mentioned convergence follows from Lemma 3.2.2 in Nualart
and the following estimate,
\begin{eqnarray*}
\mathbb{E}\left|S_{n}^{j}\left(y\right)-S_{n}^{j}\left(\widetilde{y}\right)\right|^{4} & = & \mathbb{E}\left|\sum_{k=0}^{n-1}\left(V_{j}(X_{t_{k}^{n}}^{H}(y))-V_{j}(X_{t_{k}^{n}}^{H}(\widetilde{y}))\right)\left(B_{H}^{j}(t_{k+1}^{n})-B_{H}^{j}(t_{k}^{n})\right)\right|^{4}\\
 & \leq & C\left|y-\widetilde{y}\right|^{2}
\end{eqnarray*}
for all $\left|y\right|,\left|\widetilde{y}\right|\leq n$. The equality
(\ref{eq:subs 2}) is easy to prove.
\end{proof}
The main result of this section is the following theorem. 
\begin{thm}
\label{thm:main}The process $\left\{ X_{t}^{H}(Y_{\beta}),t\in\left[0,T\right]\right\} $
satisfies the SDE (\ref{eq:wrep}).\end{thm}
\begin{proof}
It follows from the classical Kolmogorov criterion that, for each
$t\in\left[0,T\right]$, there exists a modification of the process
$\left\{ X_{t}^{H}(y),y\geq0\right\} $ that is a continuous process
whose paths are $\gamma$-Hölder for every $\gamma\in[0,\frac{1}{4})$.
Now using the equalities (\ref{eq:subs1}) and (\ref{eq:subs 2})
we obtain that the process $\left\{ X_{t}^{H}(Y_{\beta}),t\in\left[0,T\right]\right\} $
satisfies the SDE (\ref{eq:wrep}) by substituting $y=Y_{\beta}(w)$
in the SDE (\ref{eq:2}). This completes the proof.
\end{proof}

\section{Applications}

\label{sec:appl}

As an application of the SF obtained in the previous section we first
deduce, under suitable non degeneracy condition on the vector field
$\bm{V}$, the absolute continuity (with respect to the Lebesgue measure
on $\mathbb{R}^{n}$) of the law of the solution $X_{t}^{H}(Y_{\beta})$
at any time $t>0$. Secondly we give sufficient conditions for the
strict positivity of the density, cf.~Subsection~\ref{sub:abs_cont}.
Finally we derive a Gaussian mixture upper bound for the density and
its smoothness in Subsections \ref{sub:upper_bound} and \ref{sub:smooth}.

\subsection{Absolute continuity}

\label{sub:abs_cont}

In order to investigate the absolute continuity of the law of $X_{t}^{H}(Y_{\beta})$
on $\mathbb{R}^{n}$ and the strict positivity of the density we assume:
\begin{description}
\item [{(H.2)}] The vector fields $V_{0},\ldots,V_{d}$ are $C_{b}^{\infty}$.
\item [{(H.3)}] For every $x\in\mathbb{R}^{n}$ and every non vanishing
$\lambda\in\mathbb{R}^{d}$, the vector space spanned by $\{V_{j}(x),[V_{j},Z],\,1\leq j\leq d\}$
is $\mathbb{R}^{n}$, where $Z$ is given by $Z=\sum_{j=1}^{d}\lambda_{j}V_{j}$.\end{description}
\begin{prop}
Assume that Hypotheses \textup{(}\textbf{\textup{H.2}}\textup{)} and
\textup{(}\textbf{\textup{H.3}}\textup{)} hold. Then for any $t\in(0,T]$,
we have:
\begin{enumerate}
\item The law of the solution $X_{t}^{H}(Y_{\beta})$ of the SDE (\ref{eq:wrep})
has a density $p_{X_{t}^{H}(Y_{\beta})}$ with respect to the Lebesgue
measure on $\mathbb{R}^{n}$. 
\item The density $p_{X_{t}^{H}(Y_{\beta})}$ is strictly positive, that
is $p_{X_{t}^{H}(Y_{\beta})}(z)>0$ for all $z\in\mathbb{R}^{n}$. 
\end{enumerate}
\end{prop}
\begin{proof}
1. It follows from Theorem $4.3$ in Baudoin and Hairer \cite{Baudoin2007}
that, for any $y>0$ and $t\in(0,T]$, the law of the the solution
$X_{t}^{H}(y)$ of the SDE (\ref{eq:2}) has a smooth density $p_{X_{t}^{H}(y)}$
with respect to the Lebesgue measure on $\mathbb{R}^{n}$. Since $\{\bm{B}_{H}(t),0\leq t\leq T\}$
and $Y_{\beta}$ are independent, then $\{X_{t}^{H}(y),0\leq t\leq T,y>0\}$
and $Y_{\beta}$ are also independent. Now it is easy to see that
the density function of $X_{t}^{H}(Y_{\beta})$ is given by 
\begin{equation}
p_{X_{t}^{H}(Y_{\beta})}(z)=\int_{0}^{\infty}p_{X_{t}^{H}(y)}(z)M_{\beta}(y)\, dy,\, z\in\mathbb{R}^{n}.\label{eq:density-1}
\end{equation}
2. Let $t\in(0,T]$ be given. It is follows from Baudoin et al.~\cite{Baudoin:2014tw}
that, under Hypotheses (\textbf{H.2}) and (\textbf{H.3}), for any
$y>0$ the density $p_{X_{t}^{H}(y)}$ of the the solution $X_{t}^{H}(y)$
of the SDE (\ref{eq:2}) fulfills $p_{X_{t}^{H}(y)}(z)>0$, for all
$z\in\mathbb{R}^{n}$. Then for any $z\in\mathbb{R}^{n}$ we have
$p_{X_{t}^{H}(y)}(z)>0$ for all $y>0$. It follows that the density
(\ref{eq:density-1}) $p_{X_{t}^{H}(Y_{\beta})}(z)>0$ for all $z\in\mathbb{R}^{n}$.\end{proof}
\begin{rem}
The absolute continuity of the law of $X_{t}^{H}(Y_{\beta})$ may
be obtained using Theorem~8 in Nualart and Saussereau \cite{Nualart2009}
under weaker regularity conditions on $V_{j}$, $0\leq j\leq d$.
Namely, $V_{j}\in C_{b}^{3}$, $0\leq j\leq d$ and the following
non degeneracy hypothesis:
\begin{description}
\item [{(H.4)}] For every $x\in\mathbb{R}^{n}$, the vector space spanned
by $V_{1}(x),\ldots,V_{d}(x)$ is $\mathbb{R}^{n}$. 
\end{description}
\end{rem}

\subsection{Upper bound of the density}

\label{sub:upper_bound}

First of all, we recall the result of Baudoin \cite{BOT14} on the
global Gaussian upper bound for the density function $p_{X_{t}^{H}(y)}$
of the solution $X_{t}^{H}(y)$, for $y>0$. Moreover, we highlight
the dependence of $p_{X_{t}^{H}(y)}$ with respect to $y$. For this
we need to assume the same assumptions and also keep the same notation
as in the original work \cite{BOT14}. We suppose that our vector
fields $V_{1},\ldots,V_{d}$ fulfill the following antisymmetric hypothesis:
\begin{description}
\item [{(H.5)}] There exist smooth and bounded functions $\omega_{i,j}^{k}$
such that: 
\end{description}
\[
\left[V_{i},V_{j}\right]=\sum_{k=1}^{d}\omega_{i,j}^{k}V_{k}\quad\text{and}\quad\omega_{i,j}^{k}=-\omega_{i,k}^{j},\; i,j=1,\ldots,d.
\]

The following theorem is an adaptation of Theorem~1.3 in \cite{BOT14}
for the equation (\ref{eq:2}). 
\begin{thm}
Assume that Hypotheses (\textbf{H.2}), \textup{(}\textbf{\textup{H.4}}\textup{)}
and \textup{(}\textbf{\textup{H.5}}\textup{) }are satisfied. Then,
for $t\in\left(0,T\right]$, the random variable $X_{t}^{H}(y)$ admits
a smooth density $p_{X_{t}^{H}(y)}$. Furthermore, there exist $3$
positive constants $c_{t}^{(1)}(y),c_{t}^{(2)}(y),c_{t}^{(3)}(y)$
such that
\[
p_{X_{t}^{H}(y)}(z)\leq c_{t}^{(1)}(y)\exp\left(-c_{t}^{(3)}(y)\left(\left|z\right|-c_{t}^{(2)}(y)\right)^{2}\right)
\]
for any $z\in\mathbb{R}^{n}$.
\end{thm}
We would like to emphasize the dependence of the constants $c_{t}^{(1)}(y)$,
$c_{t}^{(2)}(y)$, $c_{t}^{(3)}(y)$ with respect to $y$. For that
we need a careful reading of the the proof of Theorem~1.3 in \cite{BOT14}
taking into account the dependence with respect to $y.$ In a first
step we look for the constants $c_{t}^{(2)}(y),c_{t}^{(3)}(y)$. It
should be noted that these two constants come from the tail estimate
of the solution $\mathbb{P}[X_{t}^{H}(y)>z]$. It follows from Proposition~2.2
in \cite{BOT14} (see also Hu and Nualart \cite{Hu2007}) that there
exist a constant $C>0$ depending on $V_{0}$, $\bm{V}$, $k$ and
$x_{0}$, such that 

\begin{equation}
\underset{0\leq t\leq T}{\sup}\left|X_{t}^{H}(y)\right|\leq\left|x_{0}\right|+y^{1/(2\delta)}CT\left\Vert \bm{B}_{H}\right\Vert _{0,T,\delta}^{1/\delta}\label{eq:1 Nov 6}
\end{equation}
\begin{equation}
\underset{0\leq t\leq T}{\sup}\left\Vert \gamma_{X_{t}^{H}(y)}^{-1}\right\Vert \leq\dfrac{C}{T^{2Hd}}\left[1+\exp\left(y^{1/(2\delta)}CT\left\Vert \bm{B}_{H}\right\Vert _{0,T,\delta}^{1/\delta}\right)\right]\label{eq:2 Nov 6}
\end{equation}
 
\begin{equation}
\underset{0\leq t,r_{i}\leq T}{\sup}\left|\mathbf{D}_{r_{k}}^{j_{k}}\ldots\mathbf{D}_{r_{1}}^{j_{1}}X_{t}^{H}(y)\right|\leq C\exp\left(y^{1/(2\delta)}CT\left\Vert \bm{B}_{H}\right\Vert _{0,T,\delta}^{1/\delta}\right)\label{eq:3 Nov 6}
\end{equation}
where $\mathbf{D}$ and $\gamma_{X_{t}^{H}}$ denote the Malliavin
derivative and the Malliavin matrix of $X_{t}^{H}(y)$, respectively.

On the other hand, we obtain from Theorem~3.1 in \cite{BOT14} the
following deterministic bound of the Malliavin derivative of the solution
$X_{t}^{H}(y)$, almost surely
\[
\|\mathbf{D}X_{t}^{H}(y)\|_{\infty}\leq My\exp(\theta t),\quad y>0,
\]
 where the constant $\theta$ linearly depend on $V_{0}$ and 
\[
M=\sup_{x\in\mathbb{R}^{n}}\,\sup_{\|\lambda\|\leq1}\left|\sum_{j=1}^{d}\lambda_{j}V_{j}(x)\right|^{2}.
\]
Now using the concentration property and the inequality (\ref{eq:1 Nov 6})
we obtain
\begin{equation}
\mathbb{P}[X_{t}(y)>z]\leq\exp\left(-c_{t}^{(3)}(y)\big(\left|z\right|-c_{t}^{(2)}(y)\big)^{2}\right)\label{eq:4 Nov 6}
\end{equation}
 where 
\[
c_{t}^{(2)}(y)=\sqrt{d}\,\mathbb{E}\left(\underset{i=1,\ldots,n}{\max}\left(\left|X_{t}^{H,i}(y)\right|\right)\right)
\]
 and 
\[
c_{t}^{(3)}(y)=\dfrac{1}{2dM^{2}e^{2\theta t}t^{2H}y^{2}}.
\]
For the constant $c_{t}^{(1)}(y)$, it derived from the norms of the
Malliavin derivative and the Malliavin matrix of $X_{t}^{H}(y)$.
Indeed, using the inequalities (\ref{eq:1 Nov 6})-(\ref{eq:3 Nov 6})
and the tail estimate (\ref{eq:tail-estimate_bh}), Theorem 3.14 in
\cite{BOT14} gives us the following Gaussian upper bound of the density
$p_{X_{t}^{H}(y)}$, for $y>0$,
\begin{equation}
p_{X_{t}^{H}(y)}(z)\leq c_{t}^{(1)}(y)\exp\left(-c_{t}^{(3)}(y)\left(\left|z\right|-c_{t}^{(2)}(y)\right)^{2}\right),\label{eq:5 Nov 6}
\end{equation}
where the constant $c_{t}^{(1)}(y)$ is given by 
\[
c_{t}^{(1)}(y)=\left\Vert \det\gamma_{X_{t}(y)}^{-1}\right\Vert _{L^{p}}^{m}\left\Vert \mathbf{D}X_{t}(y)\right\Vert _{k,p'}^{m'}
\]
for some constants $p,p'>1$ and integers $m,m'$. Let us note that
for $\tau<1/(128(2T)^{2(H-\delta)})$, $c_{t}^{(1)}(y)$ satisfy
\begin{eqnarray}
c_{t}^{(1)}(y) & \leq & C\left(1+t^{m/p}F(t,\frac{1}{\delta}-1,\frac{1}{\delta},Cy^{1/2\delta})^{m/p}\right)\nonumber \\
 &  & \times\left(1+t^{nH(k+1)+\frac{m'}{p'}}F(t,\frac{1}{\delta}-1,\frac{1}{\delta},Cy^{1/2\delta})^{m'/p'}\right)\label{eq:6 Nov 6}
\end{eqnarray}
 where 
\[
F\left(t,\frac{1}{\delta}-1,\frac{1}{\delta},Cy^{1/(2\delta)}\right):=\int_{0}^{+\infty}u^{(1/\delta)-1}\exp\left(-\tau u^{2}\right)\exp\left(Cty^{1/(2\delta)}u^{1/\delta}\right)du.
\]

Now we are ready to give the upper bound of the density $p_{X_{t}^{H}(Y_{\beta})}$.
\begin{prop}
Assume that Hypotheses (\textbf{H.2}), (\textbf{H.4}) and (\textbf{H.5})
are satisfied. Then for $t\in\left(0,T\right]$, the density $p_{X_{t}^{H}(Y_{\beta})}$
satisfies the following Gaussian mixture type upper bound, for all
$z\in\mathbb{R}^{n}$ 
\begin{equation}
p_{X_{t}^{H}(Y_{\beta})}(z)\leq\int_{0}^{\infty}\rho_{H}(z,y)M_{\beta}(y)\, dy,\label{eq:8 Nov 6}
\end{equation}
where 
\begin{equation}
\rho_{H}(z,y):=c_{t}^{(1)}(y)\exp\left(-c_{t}^{(3)}(y)\left(\left|z\right|-c_{t}^{(2)}(y)\right)^{2}\right).\label{eq:density1}
\end{equation}
\end{prop}
\begin{proof}
First we point out the asymptotic behavior of the function $M_{\beta}(y)$
when $y$ goes to $\infty$, see Eq.~(4.5) in \cite{Mainardi_Mura_Pagnini_2010}:
\begin{equation}
M_{\beta}(y/\beta)\sim\dfrac{1}{\sqrt{2\pi\left(1-\beta\right)}}y^{\left(\beta-1/2\right)/\left(1-\beta\right)}\exp\left(-\dfrac{1-\beta}{\beta}y^{1/(1-\beta)}\right).\label{eq:7 Nov 6}
\end{equation}
With this and the fact that $\frac{1}{2\delta}<1<\frac{1}{1-\beta}$
we see that, for any $p>0$, the integral 
\[
\int_{0}^{+\infty}\left(F\left(t,\frac{1}{\delta}-1,\frac{1}{\delta},Cy^{1/(2\delta)}\right)\right)^{p}M_{\beta}(y)\, dy
\]
 is finite. This allows us to conclude that the integral $\int_{0}^{\infty}\rho_{H}(z,y)M_{\beta}(y)\, dy$
is well defined and as a consequence the density function $p_{X_{t}^{H}(Y_{\beta})}$
satisfy the Gaussian mixture type upper bound (\ref{eq:8 Nov 6}).
\end{proof}

\subsection{Smoothness of the density}

\label{sub:smooth}

To show the smoothness of the density $p_{X_{t}^{H}(Y_{\beta})}$
we use the differentiation under the integral sign in representation
(\ref{eq:density-1}). Since $p_{X_{t}^{H}(y)}$ is smooth for any
$y>0$, then it is sufficient to obtain an upper bound of $|\partial_{\kappa}p_{X_{t}^{H}(y)}|\leq h_{\kappa}(y)$,
for any multi-index $\kappa$, such that 
\[
\int_{0}^{\infty}h_{\kappa}(y)M_{\beta}(y)\, dy<\infty.
\]
It follows from the proof of Proposition~2.1.5 in \cite{Nualart2006a}
that 
\[
|\partial_{\kappa}p_{X_{t}^{H}(y)}(z)|\leq c_{t,\kappa}^{(1)}(y)\exp\left(-c_{t}^{(3)}(y)\left(\left|z\right|-c_{t}^{(2)}(y)\right)^{2}\right),
\]
where 
\[
c_{t,\kappa}^{(1)}(y)=\left\Vert \det\gamma_{X_{t}(y)}^{-1}\right\Vert _{L^{q}}^{l}\left\Vert \mathbf{D}\, X_{t}(y)\right\Vert _{k',q'}^{l'}
\]
for some integer $l,l',k'$ and constants $q,q'>1$. The function
$c_{t,\kappa}^{(1)}$ may be estimated as in (\ref{eq:6 Nov 6}) which
implies that
\[
\int_{0}^{\infty}c_{t,\kappa}^{(1)}(y)M_{\beta}(y)\, dy<\infty.
\]
This is sufficient to guarantee the smoothness of the density $p_{X_{t}^{H}(Y_{\beta})}$.
We state this result in the following proposition.
\begin{prop}
Assume that Hypotheses (\textbf{H.2}), (\textbf{H.4}) and (\textbf{H.5})
are satisfied. Then, for $t\in\left(0,T\right]$, the density $p_{X_{t}^{H}(Y_{\beta})}$
is a smooth ($C^{\infty}$) function.
\end{prop}

\subsection*{Acknowledgments}

We would like to thank Professor David Nualart for reading the first
version of that paper and the suggestion for studying a more general
case presented here. Financial support of the project CCM - PEst-OE/MAT/UI0219/2014
and Laboratory LIBMA form the University Cadi Ayyad Marrakech are
gratefully acknowledged.


\end{document}